\documentclass{article}

\usepackage{amsmath,amsfonts,amssymb}
\usepackage{mathrsfs}
\usepackage{wasysym}
\usepackage{amsthm,thmtools}
\usepackage{graphicx}
\usepackage{color}
\usepackage{verbatim}
\usepackage{tikz-cd}
\usepackage{tikz}
\usetikzlibrary{calc,tikzmark,shapes,arrows,decorations.pathmorphing,decorations.pathreplacing, calligraphy}
\usepackage{xfrac}
\usepackage{bbm}
\usepackage[inline,shortlabels]{enumitem}
\usepackage{nccmath}		
\usepackage{xfrac}          
\usepackage[pdftex=true,hyperindex=true,pdfborder={0 0 0},colorlinks=true,allcolors=blue]{hyperref}
\usepackage{cleveref}
\usepackage[a4paper, left=2.8cm, right=2.8cm, top=3cm, bottom=2.5cm]{geometry}
\usepackage{mathtools}
\usepackage{doi}

\newtheorem{theorem}{Theorem}[section]

\newtheorem{proposition}[theorem]{Proposition}

\theoremstyle{definition}
\newtheorem{definition}[theorem]{Definition}

\newtheorem{question}[theorem]{Question}

\theoremstyle{plain}
\newtheorem{thmx}{Theorem}

\newcommand{\NN}{\mathbb{N}}
\newcommand{\ZZ}{\mathbb{Z}}

\newcommand{\FF}{\mathcal{F}}
\newcommand{\define}[1]{\textit{#1}}

\newcommand{\Addresses}{{
    \footnotesize
    \begin{samepage}
        \noindent S.~Barbieri\\
        \textsc{Departamento de Matem\'{a}tica y ciencia de la computaci\'{o}n, Universidad de Santiago de Chile, Santiago, Chile.}\\
        \indent\emph{E-mail address}: \texttt{\href{mailto:sebastian.barbieri@usach.cl}{sebastian.barbieri@usach.cl}}
    \end{samepage}
		
	\medskip

    \begin{samepage}
        \noindent L.~Poirier\\
        \textsc{Institut de Math\'ematiques de Marseille (I2M), Aix-Marseille Université, Marseille, France}\\
        \indent\emph{E-mail address}: \texttt{\href{mailto:leo.poirier@univ-amu.fr}{leo.poirier@univ-amu.fr}}
    \end{samepage}
}}

\title{A remark on inverse limits of effective subshifts}
\author{Sebasti\'an Barbieri and Leo Poirier}
\date{\today}

\begin{document}

\maketitle

\begin{abstract}
    We show that, for every finitely generated group with decidable word problem and undecidable domino problem, there exists a sequence of effective subshifts whose inverse limit is not the topological factor of any effective dynamical system. This follows from considerations on the universality under topological factors for this class of dynamical systems. 
		\medskip
		
		
		\smallskip
		
		\noindent
		\emph{MSC2020:} \textit{Primary:}
        37B10, 
		\textit{Secondary:} 20F10. 


\end{abstract}

\section{Introduction}

Given a finitely generated group $G$ and a compact metrizable space $X$, one can consider the space of all actions of $G$ on $X$ by homeomorphisms. This space can be quite large, and a natural restriction is to consider only actions which are \textit{effective}, in the sense that both the space and the dynamics can be described by an algorithm. 

Effective actions can be formally defined using computable analysis. We say that an action of a finitely generated group $G\curvearrowright X$ is an \define{effective dynamical system} (EDS), if it is topologically conjugate to an action of $G$ on a recursively compact subset of a computable metric space, in such a way that every generator of $G$ acts through computable maps. This class of actions is intimately related to other natural classes of group actions, such as subshifts of finite type, see~\cite{Hochman2009b_simulation,BarCarRoj2025}. The understanding of this relationship has led to interesting results, such as the classification of the numbers that arise as topological entropies of subshifts of finite type~\cite{HochmanMeyerovitch2010} in $\ZZ^2$ and the existence of chaotic behavior at zero-temperature for locally constant potentials~\cite{Chazottes2010}.

While the class of EDS satisfies several interesting dynamical properties, in general it is not closed under topological factor maps (see the examples in~\cite[Propositions 8.1 and 8.2]{BarCarRoj2025}). However, the class of topological factors of an EDS does satisfy a weak computability constraint: it must be topologically conjugate to the inverse limit of a (non-necessarily uniform) sequence of effective subshifts~\cite[Corollary 8.8]{BarCarRoj2025}. This led the authors of~\cite{BarCarRoj2025} to define the class of \define{weakly effective dynamical systems} (wEDS) as that of the groups actions $G\curvearrowright X$ which are topologically conjugate to the inverse limit of a sequence of EDS, or equivalently (see~\Cref{prop:wEDS_effective_subshift}), to an inverse limit of effective subshifts. The class of wEDS is stable under topological factor maps, and every EDS is a wEDS.

The main motivation behind the name wEDS in~\cite{BarCarRoj2025} was an initial belief that they might characterize the class of topological factors of EDS (\cite[Question 8.12]{BarCarRoj2025}). Our main result is that in fact there exists wEDS which are not the topological factor of any EDS in a large class of groups.

\begin{thmx}\label{mainthm:nonexistence_factor_to_weds}
    For any finitely generated group with decidable word problem and undecidable domino problem, there exists a wEDS which is not the topological factor of any EDS.
\end{thmx}

The domino problem of a finitely generated group $G$ is undecidable if there is no algorithm which, given as input a finite set of forbidden patterns, can decide whether the subshift of finite type induced by those patterns is nonempty. This notion is based on the classical result by Berger~\cite{berger_undecidability_1966} that states, in modern terms, that the domino problem is undecidable for $G=\ZZ^2$. We refer the reader to~\cite{ABJ2018} for a survey on the domino problem for groups, and note that it has been conjectured that the only groups for which the domino problem is decidable are the virtually free groups~\cite{ballier_domino_2018}.

The proof of~\Cref{mainthm:nonexistence_factor_to_weds} requires two ingredients which, from our point of view, are interesting results by themselves. These results deal with the notion of universality for topological factors. Let us fix a group $G$. We say that an EDS $G\curvearrowright \mathfrak{U}$ is \define{universal} if every other EDS of the group $G$ can be obtained as a topological factor of $G\curvearrowright \mathfrak{U}$. Similarly, a wEDS of $G$ is called \define{universal} is every other wEDS of $G$ is a topological factor of it.

The first a ingredient is a characterization, initially proved for $\ZZ^d$-actions by Hochman~\cite{Hochman2009}, of the groups with decidable word problem which admit a universal EDS.

\begin{thmx}\label{mainthm:hochmanv2}
    Let $G$ be a finitely generated group with decidable word problem. There exists a universal EDS $G\curvearrowright \mathfrak{U}$ if and only if the domino problem of $G$ is decidable.
\end{thmx}

We remark that~\Cref{mainthm:hochmanv2} is just the natural extension of Theorems 1.1 and 1.2 of~\cite{Hochman2009} to finitely generated groups with decidable word problem, and that this generalization does not present any additional difficulties. We provide this proof in~\Cref{sec:universalEDS} for the sake of completeness.

The second ingredient is the existence of universal wEDS on any finitely generated group.



\begin{thmx}\label{mainthm:universalweds}
    For every finitely generated group $G$ there exists a universal wEDS $G\curvearrowright \mathfrak{X}$.
\end{thmx}

We will now show that~\Cref{mainthm:nonexistence_factor_to_weds} follows from~\Cref{mainthm:hochmanv2} and~\Cref{mainthm:universalweds}. Suppose that every wEDS on $G$ is the topological factor of some EDS in $G$, it follows from~\Cref{mainthm:universalweds} that there exists a universal EDS on $G$. Indeed, by our assumption the universal wEDS $G\curvearrowright \mathfrak{X}$ admits an EDS extension $G\curvearrowright \mathfrak{U}$. As every wEDS (and in particular every EDS) is the topological factor of $G\curvearrowright \mathfrak{X}$, it follows that $G\curvearrowright \mathfrak{U}$ is universal.

In particular, if $G$ is a finitely generated group with decidable word problem and undecidable domino problem, the conclusion from the argument above would contradict~\Cref{mainthm:hochmanv2}, and thus we obtain~\Cref{mainthm:nonexistence_factor_to_weds}.


In~\Cref{sec:prelim} we introduce the reader to the the necessary background material and set the notation for the rest of the article. In~\Cref{sec:universalEDS} we provide the proof of~\Cref{prop:decid_DP,prop:undec_DP}  which together prove~\Cref{mainthm:hochmanv2}. In~\Cref{sec:universal_weds} we show the existence of a universal wEDS for any finitely generated group with decidable word problem, and then explain how this result is valid even without the assumption of decidable word problem. Finally, in~\Cref{sec:conclusion} we discuss the missing case of groups with decidable domino problem and give a few pointers towards a characterization of topological factors of EDS.

\subsubsection*{Acknowledgments}
The authors are grateful to N. Carrasco-Vargas for suggesting helpful improvements on our initial draft. We also thank an anonymous referee for their suggestions which greatly improved the exposition of this note. S. Barbieri was supported by the ANID project FONDECYT regular 1240085, AMSUD240026  and ECOS230003. L. Poirier acknowledges that this research was partly done in the University of Santiago de Chile during an internship and is grateful to ENS de Lyon for their funding and support.

\section{Preliminaries}\label{sec:prelim}

We denote the set of non-negative integers by $\NN$. For a finite set $A$ we write $A^*= \bigcup_{n \in \NN}A^n$ for the set of all words on $A$. A subset $L\subset A^*$ is called a \define{language}. We say $L$ is \define{recursively enumerable} if there is an algorithm which halts on input $w \in A^*$ if and only if $w \in L$. We say that $L$ is \define{decidable} if both $L$ and $A^*\setminus L$ are recursively enumerable. A partial map $g\colon A^* \to B^*$ is computable if there is an algorithm which on input $(u,v) \in A^*\times B^*$ halts if and only if $v = g(u)$.

Through the canonical binary coding we identify non-negative integers with words in $\{0,1\}^*$ and thus extend the notions of recursive enumerable and decidable to sets of non-negative integers.

A sequence of recursively enumerable languages $(L_n)_{n \in \NN}$ is \define{uniform} if there is an algorithm with in input $n \in \NN$ and a word $w$ halts if and only if $w \in L_n$.

For a finitely generated group $G$ and a finite set of generators $S\subset G$, the \define{word problem} is the language \[ \texttt{WP}_S(G) = \{w \in S^* : w \mbox{ represents the identity in } G\}.    \]
We say that a group has \define{decidable word problem} if $\texttt{WP}_S(G)$ is a decidable language. We note that this is a property of the group, in the sense that is does not depend upon the set of generators. In most of the article (with the only exception being the discussion after~\Cref{thm:universal_wEDS}) we will consider finitely generated groups with decidable word problem.

We remark that if $G$ is a group with decidable word problem, then there exists a bijection $\nu \colon \NN \to G$ with the property that the maps $(n,m) \mapsto \nu^{-1}(\nu(n)\cdot \nu(m))$ and $n \mapsto \nu^{-1}(\nu(n)^{-1})$ are computable. Consequently, we can speak about recursively enumerable and decidable subsets of $G$ through this identification and perform the group operations algorithmically.

\subsection{Subshifts and morphisms}

Given a group $G$ acting on topological spaces $X$ and $Y$ by homeomorphisms, we say that a map $f\colon X \to Y$ is a \define{morphism} if it is continuous and $G$-equivariant, that is, for every $g \in G$ and $x \in X$ we have $f(gx)=gf(x)$. We say that a morphism $f$ is a \define{topological factor map} if it is surjective, and we say it is a \define{topological conjugacy} if it is a homeomorphism.

Given a topological space $X$ and a group $G$, we endow the space $X^G$ with the product topology and the left shift map given by \[ (g\cdot x)(h) = x(g^{-1}h) \mbox{ for all } x \in X^G, g,h \in G.  \]

\begin{definition}
    A subshift is a closed and $G$-invariant subset of $A^G$, where $A$ is a finite set endowed with the discrete topology.
\end{definition}

Fix a finite set $A$. Given a finite $F\subset G$, a map $p \colon F \to A$ is called a \define{pattern} with \define{support} $F$ and its associated cylinder set is given by $[p] = \{ x \in A^G : x|_F = p\}$. A subshift can be equivalently defined as a subset of $A^G$ which avoids a collection of patterns. That is, for which there exists a collection $\FF$ of patterns such that \[ X = X_{\FF} = A^G \setminus \bigcup_{g \in G}\bigcup_{p \in \FF}g\cdot [p].   \]

We shall need the following characterization of morphisms between subshifts, for a proof see~\cite[Theorem 1.8.1]{ceccherini2010cellular}.

\begin{theorem}[Curtis-Hedlund-Lyndon]
    Let $X\subset A^G$ and $Y\subset B^G$ be subshifts and $f\colon X \to Y$ a morphism. There exists a finite $F\subset G$ and $\Phi \colon A^F\to B$ such that \[  f(x)(g) = \Phi\bigl((g^{-1}\cdot x) |_F\bigr) \mbox{ for all } x \in X, g \in G. \]
\end{theorem}

\subsection{The domino problem}

For this short subsection, fix a finitely generated group $G$ with decidable word problem. A subshift $X\subset A^G$ is called of finite type if it there exists a finite set of patterns $\FF$ such that $X=X_{\FF}$.

\begin{definition}
    We say that $G$ has decidable domino problem, if there is an algorithm which on input a finite set of patterns $\FF$ decides whether the subshift $X_{\FF}$ is nonempty.
\end{definition}

We note that for every recursively presented and finitely generated group there exists an algorithm which on input a finite set of patterns $\FF$ halts if and only if the subshift $X_{\FF}$ is empty (see~\cite[Proposition 9.3.29]{ABJ2018}). Hence we may alternatively say that $G$ has decidable domino problem, if there is an algorithm which on input a finite set of patterns $\FF$ halts if and only if the subshift $X_{\FF}$ is nonempty.

We also note that to our current knowledge, the only finitely generated groups known to have decidable domino problem are the virtually free groups. It has been conjectured that they are the only ones~\cite{ballier_domino_2018}.

\subsection{Computability in zero-dimensional spaces}

Given a finite set $A$,  we consider the Cantor space $A^{\NN}$ endowed with the product of the discrete topology. For a word $w \in A^*$ we denote by $[w]$ the cylinder set of all $x\in A^{\NN}$ which begin with $w$.

Next we shall succinctly introduce computability notions for subsets of the Cantor space. For a friendlier presentation we refer the reader to~\cite{BarCarRoj2025}.

\begin{definition}
    Let $A$ be a finite set. We say a set $X\subset A^{\NN}$ is \define{effectively closed} if there exists a recursively enumerable language $L\subset A^*$ such that \[ X = A^{\NN}\setminus \bigcup_{w \in L}[w].  \]
\end{definition}

\begin{definition}
    Let $A,B$ be finite sets. Given $X\subset A^{\NN}$, we say a map $f\colon X \to B^{\NN}$ is \define{computable} if there exists a partial computable map $g \colon A^* \to B^*$ such that for any $x \in X$, $g(x|_{\{0,\dots,n\}})$ is defined for every $n$, the cylinders $[g(x|_{\{0,\dots,n\}})]$ form a nested sequence and \[ \{f(x)\}  = \bigcap_{n \in \NN} [g(x|_{\{0,\dots,n\}})]. \]
\end{definition}

\begin{definition}
    Let $G$ be a finitely generated group. We say a left action $G\curvearrowright X$ is a zero-dimensional \define{effective dynamical system} (EDS) if it is topologically conjugate to a left action $G\curvearrowright Y$ where $Y\subset \{0,1\}^{\NN}$ is an effectively closed set and $G$ acts by computable maps.
\end{definition}

We note that there is a more general definition of EDS on which $G$ is allowed to act in more general metric spaces (see~\cite[Definition 3.25]{BarCarRoj2025}), however, the definitions are equivalent in the zero-dimensional case~\cite[Proposition 4.10]{BarCarRoj2025}, and furthermore, for recursively presented groups (in particular, for groups with decidable word problem), every EDS is the topological factor of a zero-dimensional EDS~\cite[Theorem A]{BarCarRoj2025}. Therefore, for the purpose of this note, the zero-dimensional definition given above is enough.

\begin{definition}
    A subshift $X\subset A^G$ is called \define{effective} if the left shift action $G\curvearrowright X$ is an EDS.
\end{definition}

We recall that in the case where $G$ is a finitely generated group with decidable word problem we have a natural bijection $\nu\colon \NN \to G$ that makes the standard group operations computable. Using this, we may encode a pattern $p\colon F \to A$ where $F\subset G$ is finite as a word $w \in (A\cup \{\times\})^*$, where the word $w = w_0\dots w_{n-1}$ codifies the pattern $p_w$ with support $F_w= \{\nu(i) : i \in \{0,\dots,n-1\} \mbox{ and } w_i \neq \times\}$ and is defined by $p_w(\nu(i))=w_i$. It follows that we may unambiguously speak about recursively enumerable sets of patterns in such groups. The following characterization of effective subshifts is well-known, a proof can be found in~\cite[Corollary 7.7]{BarCarRoj2025}.

\begin{proposition}
    Let $G$ be a finitely generated group with decidable word problem. A subshift $X\subset A^G$ is effective if and only if there exists a recursively enumerable set of forbidden patterns $\FF$ such that \[ X = A^G\setminus \bigcup_{p \in \FF}g[p].   \]
\end{proposition}

A similar description can be given for general zero-dimensional EDS, albeit, with an infinite alphabet. Again, if $G$ is a finitely generated group with decidable word problem, the set $G\times \NN$ can be identified by a computable bijection with $\NN$, and thus we may speak about effectively closed subsets of $A^{G\times \NN}$. We endow the space $A^{G\times \NN}$ with the left shift action given by \[ \bigl(g\cdot y\bigr)(h,n) = y(g^{-1}h,n) \mbox{ for all } y \in Y, g,h \in G \mbox{ and } n \in \NN.   \]

Notice that the left shift action $G\curvearrowright A^{G\times \NN}$ is computable. The following characterization of zero-dimensional EDS is quite useful as it encodes the computability of an arbitrary action in the topology of the space. We give a brief proof sketch.

\begin{proposition}\label{prop:orbitrep}
    Let $G$ be a finitely generated group with decidable word problem. An action $G\curvearrowright X$ is a zero-dimensional EDS if and only if it is topologically conjugate to the shift action $G\curvearrowright Y$ for some effectively closed and $G$-invariant $Y\subset \{0,1\}^{G\times \NN}$.
\end{proposition}

\begin{proof}
    The left shift action is computable in $\{0,1\}^{G\times \NN}$ and thus if $G\curvearrowright X$ is topologically conjugate to the left shift action $G\curvearrowright Y$ for some effectively closed $Y\subset \{0,1\}^{G\times \NN}$, it follows that it is an EDS.

    Conversely, if $G\curvearrowright X$ is an EDS, it is topologically conjugate to a left action $G\curvearrowright Z$ where $Z\subset \{0,1\}^{\NN}$ and $G$ acts by computable maps. Consider the orbit space $Y\subset \{0,1\}^{G\times \NN}$, where $y \in Y$ if and only if $\bigl(y(1_G,n)\bigr)_{n \in \NN} \in Z$ and  $\bigl(y(g,n)\bigr)_{n \in \NN} = g^{-1} \cdot \bigl(y(1_G,n)\bigr)_{n \in \NN}$ for all $g \in G$.  
    It is clear that $Y$ is $G$-invariant and that $G\curvearrowright Z$ is topologically conjugate to the shift action of $G$ on $Y$. Finally, from the fact that $Z$ is effectively closed and that $G$ acts by computable maps one can deduce that $Y$ is effectively closed.\end{proof}



\subsection{Weak effective dynamical systems and universality}

Let $(G \curvearrowright X_n)_{n \in \NN}$ be a sequence of actions of a group $G$ on topological spaces $X_n$ by homeomorphisms and a sequence $(\pi_n)_{n \in \NN}$ of topological factor maps $\pi_{n}\colon X_{n+1} \to X_{n}$. The \define{inverse limit} associated to these sequences is the space \[ \lim_{\leftarrow} X_n = \{ (x_n)_{n \in \NN} \in \prod_{n \in \NN} X_n : x_n = \pi_{n}(x_{n+1}) \mbox{ for every } n \in \NN \},     \]
endowed with the pointwise action of $G$.

\begin{definition}
    An action of a group $G$ on a zero-dimensional space is called a \define{weak effective dynamical system} (wEDS) if it is topologically conjugate to an inverse limit of EDS.
\end{definition}

We remark that in~\cite{BarCarRoj2025} a wEDS is defined as the inverse limit of effective subshifts rather than of arbitrary EDS. These two notions are equivalent.

\begin{proposition}\label{prop:wEDS_effective_subshift}
    A $G$-action is a wEDS if and only if it topologically conjugate to an inverse limit of effective subshifts.
\end{proposition}

\begin{proof}
    Effective subshifts are EDS, so naturally every inverse limit of subshifts is a wEDS. For the converse, we employ a diagonal argument. Let $(G\curvearrowright X_n)_{n \in \NN}$ be a sequence of EDS. By~\Cref{prop:orbitrep}, we can without loss of generality assume $G\curvearrowright X_n$ is the shift action on an effectively closed and shift invariant $X_n \subset \{0,1\}^{G\times \NN}$. Let $\pi_{n}\colon X_{n+1} \to X_{n}$ be factor maps and let $X$ be the inverse limit associated to these sequences. 

    Consider $Z \subset \{0,1\}^{G\times \NN^2}$ as the set of maps $z\colon G\times \NN^2\to \{0,1\}$ such that if we let $x_n \in \{0,1\}^{G\times \NN}$ be given by $x_n(g,i) =z(g,i,n)$, then for each $n \in \NN$ we have $x_n \in X_n$ and $\pi_n(x_{n+1})=x_n$. Endow $Z$ with the natural shift $G$-action. It is clear from the definition that $G\curvearrowright Z$ is topologically conjugate to the endowed pointwise $G$-action on $X$.
    
    For each $t \in \NN$, define $Y_t$ as the set \begin{align*}
        Y_t & = \{ y \colon G \to \{0,1\}^{\{0,\dots,t\}^2} : \mbox{ there is } z \in Z: \bigl(y(g)\bigr)(i,n) = z(g,i,n)\}\\
        & = \{ y \colon G \to \{0,1\}^{\{0,\dots,t\}^2} : \mbox{ there is } x \in X_{t+1}: \bigl(y(g)\bigr)(i,n) = \bigl(\pi_n \circ \dots \circ \pi_{t}(x)\bigr)(g,i)\}
    \end{align*}   

    From the second equality and the fact that $X_{t+1}$ is an EDS it follows that $Y_t$ is effectively closed. As $Z$ is shift invariant, it follows that so is $Y_t$ and thus the shift action $G\curvearrowright Y_t$ is an effective subshift. Finally, if we endow the sequence $(G\curvearrowright Y_t)$ with the natural factor maps that restrict the maps in each coordinate, we have that the pointwise shift $G$-action on the inverse limit $\lim_{\leftarrow} Y_t$ is topologically conjugate to $G\curvearrowright Z$.\end{proof}

We note that, as a consequence of~\Cref{prop:orbitrep} every zero-dimensional EDS can be represented as an inverse limit of effective subshifts, however, it has the added property that the sequence of effective subshifts $(X_n)_{n \in \NN}$ is \define{uniform}, that is, there is a single algorithm which on input $n$ produces a list of forbidden patterns for $X_n$ and its associated factor map. In our definition of wEDS we do not ask for the sequences to be uniform.

 \begin{definition}
        We say that an EDS (resp. wEDS) $G\curvearrowright \mathfrak{U}$ is \define{universal} if for every EDS (resp. wEDS) $G\curvearrowright X$ there exists a topological factor map $f\colon \mathfrak{U}\to X$.
    \end{definition}

\section{Existence of universal EDS}\label{sec:universalEDS}

The purpose of this section is to prove~\Cref{mainthm:hochmanv2}. In~\Cref{prop:decid_DP} we show that a universal EDS exists in any group with decidable domino problem, while in~\Cref{prop:undec_DP} we show that on any group with decidable word problem, the existence of a universal EDS implies the decidability of the domino problem. \Cref{mainthm:hochmanv2} follows from putting both propositions together.

\begin{proposition}\label{prop:decid_DP}
    Let $G$ be a finitely generated group with decidable word and domino problems. There exists a universal EDS $G\curvearrowright \mathfrak{U}$.
\end{proposition}

\begin{proof}

    Let $(L_n)_{n \in \NN}$ be a recursive enumeration of all recursively enumerable subsets of patterns of the form $\{0,1\}^F$, where $F$ ranges over all finite subsets of $G \times \NN$. In this way, every zero-dimensional EDS for some $n \in \NN$ is topologically conjugate to the shift action on \[ X_n = \{0,1\}^{G\times \NN}\setminus \bigcup_{g \in G}\bigcup_{p \in L_n}g[p].   \]

    Consider a recursive enumeration $p_0,p_1,\dots$ of all patterns of the form $\{0,1\}^F$, where $F$ ranges over all finite subsets of $G \times \NN$. For each $n \in \NN$, we construct a sublanguage $L'_n \subset L_n$ as follows:

    For each $k \in \NN$, let $T_k$ be the set of patterns $p\in \{p_0,\dots,p_k\}$ such that the algorithm that recognizes $L_n$ halts on input $p$ on at most $k$ steps. Let $F_k\times I_k$ be the smallest finite subset of $G\times \NN$ such that the support of every pattern in $T_k$ is contained in $F_k\times I_k$. Let $P_k$ be the set of all maps $w\colon F_k \to \{0,1\}^{I_k}$ for which there is $p \in T_k$ such that for every $(g,i)\in \operatorname{supp}(p)$ we have $w(g)(i)=p(g,i)$.

    Consider the subshift of finite type $Z_{n,k} \subset (\{0,1\}^{I_k})^G$ given by \[ Z_{n,k} = (\{0,1\}^{I_k})^G\setminus \bigcup_{g \in G}\bigcup_{w \in P_k}g[w].  \]
    We take $L'_n$ as the union of $P_k$ over all $k$ such that the corresponding $Z_{n,k}$ is nonempty.

    Notice that there is an algorithm that given the algorithm that recognizes $L_n$ and an integer $k$, computes the sets $T_k$, $F_k \times I_k$ and $P_k$. Furthermore, as the domino problem on $G$ is decidable, there is an algorithm which on finite time decides whether $Z_{n,k}$ is empty, thus it follows that $L'_n$ is recursively enumerable. Moreover, as this algorithm is uniform for all $n$, we have that $(L'_n)_{n \in \NN}$ is a uniform sequence of recursively enumerable languages.
    
    Let \[ X'_n = \{0,1\}^{G\times \NN}\setminus \bigcup_{g \in G}\bigcup_{p \in L'_n}g[p].   \]

    By our construction, it follows that $X'_n$ is always nonempty. Furthermore, we have that $L_n = L'_n$ if and only if $X_n \neq \varnothing$, in which case $X_n = X'_n$.

    Let $\mathfrak{U}=\prod_{n \in \NN}X'_n$ and take $G\curvearrowright \mathfrak{U}$ as the coordinate-wise shift. We have that $\mathfrak{U}$ is nonempty and every $G$-EDS is the topological factor of $G\curvearrowright \mathfrak{U}$ by a projection map.

    Finally, $G\curvearrowright \mathfrak{U}$ is an EDS. To see this, for a pattern $p \in \{0,1\}^{F\times I}$, denote by $\iota_n(p)$ the pattern with support $F\times (I\times \{n\})$ such that $\iota_n(p) (g,(i,n))= p(g,i)$ for every $g\in F$ and $i \in I$. We have that $G\curvearrowright \mathfrak{U}$ is topologically conjugate to the shift action on
\[  \mathfrak{X} = \{0,1\}^{G \times \NN^2}\setminus \bigcup_{n \in \NN}\bigcup_{g \in G}\bigcup_{p \in L'_n}g[\iota_n(p)]. \]

Where the conjugacy $\phi\colon \mathfrak{U}\to \mathfrak{X}$ is given by $\phi( (x_n)_{n \in \NN})(g,i,j) = (x_j)(g,i)$. As $(L'_n)_{n \in \NN}$ is a uniform sequence of recursively enumerable languages, it follows that $\mathfrak{U}$ is effectively closed.\end{proof}

\begin{proposition}\label{prop:undec_DP}
   Let $G$ be a group with decidable word problem and suppose that there exists a universal EDS $G\curvearrowright \mathfrak{U}$. Then the domino problem of $G$ is decidable.
\end{proposition}


\begin{proof}
   Assume without loss of generality that $\mathfrak{U}\subset \{0,1\}^{G \times \NN}$ and that $G\curvearrowright \mathfrak{U}$ is the left shift map. Let $(u_i)_{i \in \NN}$ be a recursively enumerable sequence of maps from finite subsets of $G \times \NN$ to $\{0,1\}$ such that \[ \mathfrak{U} = \{0,1\}^{G\times \NN}\setminus \bigcup_{i \in \NN}[u_i].\] Let $\FF$ be a finite set of patterns over the alphabet $A=\{1,\dots,n\}$ and denote by $X_{\FF}$ its corresponding $G$-SFT. Without loss of generality, by enlarging our finite set of forbidden patterns, we may take a finite subset $W\subset G$ and suppose that every pattern from $\FF$ has support $W$.
    
    Let $(C^t_1,C^t_2,\dots C^t_n)_{t \in \NN}$ be a recursive enumeration of all $n$-tuples of clopen subsets of $\{0,1\}^{G\times \NN}$ such that $(C^t_i)_{i=1}^n$ is a partition of $\{0,1\}^{G\times \NN}$.

    We run the following algorithm which iterates over all tuples $(m,t) \in \NN\times \NN$:
    \begin{enumerate}
        \item Take $\mathfrak{U}_m = \{0,1\}^{G \times \NN}\setminus \bigcup_{i =1}^m[u_i]$. Construct the set \[ P_{m,t} =  \left\{\varphi \colon W \to A : \mathfrak{U}_m\cap \bigcap_{g \in W}g \cdot C^t_{\varphi(g)}   \neq \varnothing \right\}.   \]
        \item If $P_{m,t}\cap \FF = \varnothing$, halt. Otherwise, try the next tuple.
    \end{enumerate}

    Suppose the algorithm halts, that is, there is $(m,t)$ such that $P_{m,t}\cap \FF = \varnothing$. In this case, we can define the map $f\colon \mathfrak{U} \to A^G$ given for $g\in G$ and $x \in \mathfrak{U}$ by the relation \[ f(x)(g) = i \mbox{ if and only if } x \in g\cdot C^t_i.   \] 
    As $(C^t_i)_{i=1}^n$ is a partition of $\{0,1\}^{\NN}$, the map $f$ is well defined and continuous. Furthermore, if we equip $A^G$ with the shift action, we have that $f$ is $G$-equivariant, namely, for every $g \in G$ we have that $f(gx) = g(f(x))$. Moreover, as $\mathfrak{U}$ is nonempty, we have that $f(\mathfrak{U})$ is a nonempty subshift. Finally, as $\mathfrak{U}\subset \mathfrak{U}_m$ and $P_{m,t}\cap \FF = \varnothing$, we obtain that for each $x\in \mathfrak{U}$, we have $f(x)|_W \notin \FF$, and thus $f(\mathfrak{U})\subset X_{\FF}$, which is thus nonempty.

    Now suppose that $X_{\FF}$ is nonempty. As every $G$-SFT is a $G$-EDS, it follows that there exists a factor map $f\colon \mathfrak{U}\to X_{\FF}$. Consider for each $i \in A$ the partition $(U_i)_{i=1}^n$ given by \[U_i = \{x \in \mathfrak{U} : f(x)(1_G) = i\}.\]
    It follows that these sets are clopen, and thus we can find $t \in \NN$ such that $U_i = C^t_i\cap \mathfrak{U}$. By compactness, as $\mathfrak{U} = \bigcap_{m \in \NN}\mathfrak{U}_m$, there exists $m$ such that for every $\varphi\colon W \to A$ \[ \mathfrak{U}_m\cap \bigcap_{g \in W} g \cdot C^t_{\varphi(g)}   = \varnothing \mbox{ if and only if }  \mathfrak{U} \cap \bigcap_{g \in W} g\cdot C^t_{\varphi(g)}  = \varnothing. \]
    Hence if $\varphi \in P_{m,t}$, we have that $\mathfrak{U}\cap \bigcap_{g \in W} g \cdot C^t_{\varphi(g)}  \neq \varnothing$. Choosing $x \in \mathfrak{U} \cap \bigcap_{g \in W} g\cdot C^t_{\varphi(g)} = \bigcap_{g \in W}g\cdot U_{\varphi(g)}$ we get $f(x)|_W = \varphi$. Furthermore, as $f(x) \in X_{\FF}$, we get that $\varphi = f(x)|_W \notin \FF$ and thus  $P_{m,t} \cap \FF = \varnothing$.
    
    We conclude that this algorithm halts if and only if $X_{\FF}$ is nonempty and thus the domino problem of $G$ is decidable.
    \end{proof}

    \section{Existence of universal wEDS}\label{sec:universal_weds}

    In this section we shall show that every finitely generated group admits a universal wEDS. Intuitively, this universal object is obtained by pasting together all possible directed sequences of effective subshifts in such a way that the resulting object is also an inverse limit. This is achieved by identifying $\NN$ with the nodes of an unrooted universal directed tree with countably many leaves where each node carries an effective subshift and each edge carries a topological factor map between the subshifts at the nodes. In this way, every possible (infinite) sequence of subshifts and factor maps appears as an infinite path in the tree, from which we can extract a factor to any given inverse limit of effective subshifts. We remark that in this proof we do not require the indexation by the tree to be computable, the only computable objects are the subshifts on the nodes and their factor maps.

    For ease of exposition, we will first do the proof in the case of a group with decidable word problem and then argue how to generalize it.
    \begin{theorem}\label{thm:universal_wEDS}
        For every finitely generated group with decidable word problem there exists a universal wEDS $G\curvearrowright \mathfrak{X}$.
    \end{theorem}

    \begin{proof}
        Consider the collection of all tuples of the form $(h,s,X,Y,f)$ where $h$ and $s$ are non-negative integers, $X$ and $Y$ are nonempty effective subshifts on $G$ whose alphabets are finite subsets of $\NN$ and $f$ is a topological factor map between $X$ and $Y$.

        There are at most countably many effective subshifts (with their alphabets identified with finite subsets of $\NN$) for any given group $G$. Moreover, by the Curtis-Hedlund-Lyndon theorem, there can be at most countably many topological morphisms between $X$ and $Y$ and thus at most countably many topological factor maps. Therefore the space of all such tuples is countable.
        
        Notice that this set of tuples is nonempty as the trivial subshift with a single constant configuration is always effective and the identity map is an automorphism. Thus, the set of all such tuples is countably infinite. Let $\varphi$ be a bijection from $\NN$ to the space of all such tuples. Given $n \in \NN$, write $\varphi(n) = (h_n,s_n,X_n,Y_n,f_n)$.
        Consider the space \[ \mathfrak{X}' = \prod_{n \in \NN} X_n, \]
        It is clear that $\mathfrak{X}'$ can be identified with a compact subset of the Baire space $\NN^{\NN}$, and furthermore, that $G$ acts naturally on $\mathfrak{X}'$ through the coordinate-wise shift action.

        Let $\mathfrak{X}\subset \mathfrak{X}'$ be the set of all $(x_n)_{n \in \NN}\in \mathfrak{X}'$ which satisfy for every $n 
        \in \NN$ condition $(\mathtt{C}_n)$: \begin{equation}
            \mbox{ if } Y_n = X_{s_n} \mbox{ and } h_n = h_{s_n}+1, \mbox{ then } x_{s_n} = f_n(x_n).\tag{$\mathtt{C}_n$ }
        \end{equation}

        Let us introduce some notation that will be useful in the remainder of the proof. For $n \in \NN$, we write $D_0(n) = n$ and define iteratively for $k \geq 1$, $D_k(n)= s_{D_{k-1}(n)}$. We call $D_k(n)$ the $k$-th descendant of $n$. We also define the height $H(n)$ as the largest value of $k$ such that for $0 \leq i < k$ we have $h_{D_i(n)} = h_{D_{i+1}(n)}+1$. We remark that, by vacuity, if $h_n=0$ then $H(n)=0$ and also that $H(n) \leq h_n$.
        
        For $k \in \NN$, consider the set of integers \[ A_k = \{ D_i(n) : n \in \{0,\dots,k\}, i \in \{0,\dots,H(n)\} \}.    \]
        That is, we take all non-negative integers $n \leq k$ and define $A_k$ to be the set of all their descendants up to their height. Notice that each set $A_k$ is finite, and that $(A_k)_{k \in \NN}$ forms a monotone increasing sequence of sets with $\bigcup_{k \in \NN}A_k = \NN$.

        Let us argue that $\mathfrak{X}$ is a wEDS. Consider the projection map $\pi_k \colon \mathfrak{X} \to \prod_{n\in A_k} X_n$ given by $\pi_k((x_n)_{n\in \NN}) = (x_n)_{n\in A_k}$ and let $\mathfrak{X}_k = \pi_k(\mathfrak{X})$. It is clear that $\mathfrak{X}_k$ is a subshift as it is a closed $G$-invariant subset of $\prod_{n\in A_k} X_k$. Furthermore, if we define for $i \in \NN$ the natural projection map $g_{i}\colon \mathfrak{X}_{i+1} \to \mathfrak{X}_i$, then $\mathfrak{X}$ is the inverse limit of the sequence $(\mathfrak{X}_k)_{k \in \NN}$ endowed with the collection of factor maps $g_{i}$. Therefore, in order to show that $\mathfrak{X}$ is a wEDS, we only need to argue that for every $k \geq 0$, the subshift $\mathfrak{X}_k$ is effective. 
        Indeed, let $k \geq 0$ and note first that $\prod_{n\in A_k}X_k$ is a nonempty product of effective subshifts and thus a nonempty effective subshift itself. Furthermore, by our definition of $A_k$, it follows that $\mathfrak{X}_k$ is precisely the set of elements of $\prod_{n\in A_k}X_k$ which satisfy conditions $\mathtt{C}_n$ for each $n \in A_k$. Noting that each of these conditions can be implemented with finitely many forbidden patterns (as a consequence of the Curtis-Hedlund-Lyndon theorem), we conclude that $\mathfrak{X}_k$ is effective. 

        Finally, let us show that $G\curvearrowright \mathfrak{X}$ is universal. Let $G\curvearrowright X$ be a wEDS, then it is topologically conjugate to the inverse limit of a sequence of effective subshifts. Let $(Y_i)_{i \in \NN}$ be said sequence of effective subshifts and for $i \in \NN$ let $t_{i}\colon Y_{i+1}\to Y_{i}$ be the associated factor maps.

        Construct iteratively a sequence $(n_i)_{i \in \NN}$ such that $\varphi(n_0) = (0,0,Y_0,Y_0,\operatorname{Id}_{Y_0})$ and such that for $i \geq 1$ we have $\varphi(n_i) = (i, n_{i-1},Y_i,Y_{i-1}, t_{i-1})$. Consider the map $\xi\colon \mathfrak{X}\to \prod_{i \in \NN} Y_i$ given by \[ \xi( (x_n)_{n \in \NN}) = (x_{n_i})_{i \in \NN}.   \]
        It is obvious that $\xi$ is continuous and $G$-equivariant. Moreover, for each $i \in \NN$ we have $x_{n_i}\in X_{n_i} = Y_i$. Furthermore, by condition $(\mathtt{C}_{n_i})$ we have that for $i \geq 1$, $t_i(x_{n_i})=f_{n_i}(x_{n_i})=x_{n_{i-1}}$, hence the image of any point in $\mathfrak{X}$ lies in the inverse limit of the $Y_i$. It remains to show that $\xi$ is surjective.

        Let $y = (y_i)_{i \in \NN}$ such that for $i \in \NN$ we have $t_i(y_{i+1})=y_{i}$. We will construct a sequence $x=(x_n)_{n \in \NN}$ such that $x \in \mathfrak{X}$ and $\xi(x)=y$. To that end, we first set $x_{n_i}=y_i$ for each $i \in \NN$. Then iteratively for each $n \in \NN$ we do the following: we compute the largest value (if it exists) of $i \in \{0,\dots,H(n)\}$ for which $x_{D_i(n)}$ is not yet defined.
        \begin{enumerate}
            \item If $x_{D_i(n)}$ is defined for all such $i$, we do nothing and move on to $n+1$.
            \item If $i < H(n)$, we choose $x_{D_i(n)} \in X_{D_i(n)}$ such that $f_{D_i(n)}(x_{D_i(n)}) = x_{D_{i+1}(n)}$ and repeat the process with $i-1$ until all the values $x_{D_i(n)}$ are defined.
            \item If $i = H(n)$, then we choose $x_{D_{H(n)}(n)} \in X_{D_{H(n)}(n)}$ arbitrarily and then proceed as above.
        \end{enumerate} 

        Notice that at any step during this process, if $x$ is defined at a coordinate $n$, then it is also defined in all descendants of $n$ up to $H(n)$. Furthermore, all coordinates which are defined at a given step satisfy condition $\mathtt{C}_n$, and every coordinate is eventually defined. It follows that $x\in \mathfrak{X}$ and clearly $\xi(x)=y$.\end{proof}

        Let us now consider the situation of a group $G$ which does not have decidable word problem.  In this general setting, we may no longer characterize effective subshifts through forbidden patterns on $G$, but we may identify them in such a way lifting the action of $G$ to a free group~\cite[Proposition 7.1]{BarCarRoj2025}. More precisely, if we let $S\subset G$ be a finite set of generators and consider the free group $F(S)$ generated by $S$ and the canonical epimorphism $\phi \colon F(S) \to G$, then a subshift $X\subset A^G$ is effective if and only if its pullback $\widehat{X}\subset A^{F(S)}$ is effective, where \[ \widehat{X} = \{ y \in A^{F(S)} : \mbox{ there exists } x \in X, y(w) = x(\phi(w)) \mbox{ for all } w \in F_2, \}.     \]
        
        Thus in the proof of~\Cref{thm:universal_wEDS} we may consider instead of $G$ the group $F(S)$ and instead of effective subshifts $X_n$ and $Y_n$ on the group $G$, we replace them by their pullbacks $\widehat{X}_n$ and $\widehat{Y}_n$ which are subshifts on $F(S)$. The rest of the proof is exactly the same and at the end we note that the kernel of the action $F(S) \curvearrowright \mathfrak{X}$ contains $\operatorname{ker}(\phi)$ and thus the resulting induced action of $G$ is the universal $G$-wEDS. We thus obtain~\Cref{mainthm:universalweds}.

\section{Final remarks}\label{sec:conclusion}

\Cref{mainthm:nonexistence_factor_to_weds} shows that on finitely generated groups with decidable word problem but undecidable domino problem, the class of wEDS does not characterize the class of topological factors of EDS. Two questions arise naturally.

\begin{question}
    Let $G$ be a group with decidable word problem and undecidable domino problem. Is there an additional invariant of recursive nature that characterizes which wEDS are topological factors of EDS?
\end{question}

We remark that if an inverse limit of effective subshifts is uniform (that is, there is a unique algorithm which describes all subshifts in the sequence and factor maps) then it is an EDS, and all zero-dimensional EDS can be represented in this manner~\cite[Section 4]{BarCarRoj2025}. In particular, any subsequence of such a uniform sequence gives rise to a wEDS factor and thus all topological factors of EDS arise in this way. However, we do not see any natural way to state this restriction in purely computable terms.

The second question is whether the class of wEDS is the right characterization for topological factors of EDS on groups in which both the word problem and the domino problem are decidable, such as $\ZZ$.

\begin{question}\label{ques:2}
    Let $G$ be a group with decidable word problem and decidable domino problem. Is every wEDS on $G$ a topological factor of some EDS? 
\end{question}

Naturally, if the answer to~\Cref{ques:2} is positive, then a universal wEDS would admit a universal EDS for said group as an extension. This hints that if the answer is positive, it should be possible to ``effectivize'' the construction in the proof of~\Cref{mainthm:universalweds} using elements from the proof of~\Cref{prop:decid_DP}. We were not able to do this, the main difficulty with such an approach is that, at least in principle, one can no longer enumerate all topological factor maps but merely all topological morphisms. 

Finally, we will argue that in Theorems \ref{mainthm:nonexistence_factor_to_weds},~\ref{mainthm:hochmanv2} and~\ref{mainthm:universalweds} we may in fact drop the hypothesis that the group is finitely generated and extend it to countable groups, but we must pay the price that now the underlying notions of EDS, word problem and domino problem will depend upon the choice of generating set. 

Let us be more precise. Let $
\mathcal{S}= (g_n)_{n \geq 0}$ be a sequence of elements of $G$ which generates it. We can define the word problem $\texttt{WP}_{\mathcal{S}}(G)$ as the set of $w_0\dots w_k \in \NN^*$ such that $g_{w_0}\cdots g_{w_k}$ represents the identity in $G$. We note that unlike finitely generated groups, in general a countable group decidability of the word problem depends heavily on the choice of $\mathcal{S}$, see for instance examples 5.3 and 5.4 of~\cite{barbierilemp:tel-01563302}. Similarly, an action $G\curvearrowright X$ is uniformly computable, if the collection of maps $(g_n\colon X \to X)_{n \geq 0}$ is uniformly computable, that is, there is an algorithm which on input $n \in \NN$ outputs a description of $g_n$. Again, this depends upon the chosen enumeration of $\mathcal{S}$, thus, by modding out topological conjugacy, we may only speak about an effective dynamical system of $G$ with respect to $\mathcal{S}$. Similarly, we can use $\mathcal{S}$ to encode patterns and speak about the domino problem of a group with respect to $\mathcal{S}$. 

The proofs of Theorems \ref{mainthm:nonexistence_factor_to_weds},~\ref{mainthm:hochmanv2} and~\ref{mainthm:universalweds} can naturally be extended to this setting with minimal modifications, as we did not use finite generation in any meaningful way besides the fact that if a finite set of generators acts by computable maps, then the action of $G$ is uniformly computable.

\bibliographystyle{plainurl}
\bibliography{bibliography}

@article{BarCarRoj2025,
    AUTHOR = {Barbieri, Sebasti\'{a}n and Carrasco-Vargas, Nicanor and Rojas,
              Crist\'{o}bal},
     TITLE = {Effective dynamical systems beyond dimension zero and factors
              of {SFT}s},
   JOURNAL = {Ergodic Theory Dynam. Systems},
  FJOURNAL = {Ergodic Theory and Dynamical Systems},
    VOLUME = {45},
      YEAR = {2025},
    NUMBER = {5},
     PAGES = {1329--1369},
      ISSN = {0143-3857},
   MRCLASS = {37B10 (03D78 20F10 37B02)},
  MRNUMBER = {4887199},
       DOI = {10.1017/etds.2024.79},
       URL = {https://doi.org/10.1017/etds.2024.79},
}

@phdthesis{barbierilemp:tel-01563302,
TITLE = {{Shift spaces on groups : computability and dynamics}},
AUTHOR = {Barbieri, Sebasti{\'{a}}n},
URL = {https://tel.archives-ouvertes.fr/tel-01563302},
NUMBER = {2017LYSEN021},
SCHOOL = {{Universit{\'e} de Lyon}},
YEAR = {2017},
MONTH = Jun,
TYPE = {Theses},
HAL_ID = {tel-01563302},
HAL_VERSION = {v1},
}

@article{ballier_domino_2018,
	title = {The domino problem on groups of polynomial growth},
	volume = {12},
	issn = {1661-7207},
	doi = {10.4171/GGD/439},
	language = {English},
	number = {1},
	journal = {Groups, Geometry, and Dynamics},
	author = {Ballier, Alexis and Stein, Maya},
	year = {2018},
	pages = {93--105},
}

@book{berger_undecidability_1966,
	series = {Memoirs of the {American} {Mathematical} {Society}},
	title = {The undecidability of the domino problem},
	isbn = {978-0-8218-1266-2 978-1-4704-0013-2},
	url = {http://www-ams-org/memo/0066},
	language = {en},
	urldate = {2021-09-04},
	publisher = {American Mathematical Society},
	author = {Berger, Robert},
	year = {1966},
}

@article{Chazottes2010,
  title = {On the Zero-Temperature Limit of Gibbs States},
  volume = {297},
  ISSN = {1432-0916},
  url = {http://dx.doi.org/10.1007/s00220-010-0997-8},
  DOI = {10.1007/s00220-010-0997-8},
  number = {1},
  journal = {Communications in Mathematical Physics},
  publisher = {Springer Science and Business Media LLC},
  author = {Chazottes,  Jean-René and Hochman,  Michael},
  year = {2010},
  month = mar,
  pages = {265–281}
}

@Article{Hochman2009b_simulation,
  Title                    = {On the dynamics and recursive properties of multidimensional symbolic systems},
  Author                   = {Hochman, Mike},
  Journal                  = {Inventiones Mathematicae},
  Year                     = {2009},
  Number                   = {1},
  Pages                    = {131--167},
  Volume                   = {176}
}

@article{Hochman2009,
title = {A note on universality in multidimensional symbolic dynamics},
journal = {Discrete and Continuous Dynamical Systems - S},
volume = {2},
number = {2},
pages = {301-314},
year = {2009},
issn = {1937-1632},
doi = {10.3934/dcdss.2009.2.301},
author = {Michael Hochman}
}

@Article{HochmanMeyerovitch2010,
  Title                    = {A characterization of the entropies of multidimensional shifts of finite type},
  Author                   = {Hochman, Mike and Meyerovitch, Tom},
  Journal                  = {Annals of Mathematics},
  Year                     = {2010},
  Number                   = {3},
  Pages                    = {2011--2038},
  Volume                   = {171}
}

@incollection{ABJ2018,
    AUTHOR = {Aubrun, Nathalie and Barbieri, Sebasti\'{a}n and Jeandel,
              Emmanuel},
     TITLE = {About the domino problem for subshifts on groups},
 BOOKTITLE = {Sequences, groups, and number theory},
    SERIES = {Trends Math.},
     PAGES = {331--389},
 PUBLISHER = {Birkh\"{a}user/Springer, Cham},
      YEAR = {2018},
   MRCLASS = {20F10 (03D35 37B50 68Q05)},
  MRNUMBER = {3799931},
MRREVIEWER = {Ivan V. Latkin},
       DOI = {10.1007/978-3-319-69152-7\_9},
}

@Book{ceccherini2010cellular,
  author    = {Ceccherini-Silberstein, T. and Coornaert, M.},
  publisher = {Springer},
  title     = {Cellular Automata and Groups},
  year      = {2010},
  doi       = {10.1007/978-3-642-14034-1},
}

\bigskip
\Addresses

\end{document}